\documentclass[12pt]{article}
\usepackage{amsgen, amsmath, amsfonts, amsthm, amssymb,
enumerate,amscd}
\usepackage[all]{xy}
\usepackage{srcltx}

\usepackage{color}

\definecolor{blue}{rgb}{0,0,1}
\definecolor{red}{rgb}{1,0,0}
\definecolor{green}{rgb}{0,1,0}
\definecolor{purple}{rgb}{1,0,1}

\long\def\red#1\endred{\textcolor{red}{#1}}
\long\def\blue#1\endblue{\textcolor{blue}{#1}}
\long\def\purple#1\endpurple{\textcolor{purple}{#1}}
\long\def\green#1\endgreen{\textcolor{green}{#1}}

\newtheorem{defn}{Definition}[section]
\newtheorem{prop}[defn]{Proposition}
\newtheorem{lem}[defn]{Lemma}
\newtheorem{thm}[defn]{Theorem}
\newtheorem{cor}[defn]{Corollary}

\newcommand {\G}{{\Gamma}}

\newcommand {\g}{{\gamma}}

\newcommand {\ca}{{\mathbf a}}

\def\mod{\operatorname{mod}}

\def\ca{{\mathfrak a}}

\def\sb{{\sigma_\mathfrak b}}

\begin{document}

\title{Shifted convolution sums and Eisenstein series formed with
modular symbols}
\author{Roelof Bruggeman \\Nikolaos Diamantis}

\maketitle

\section{Introduction}

Shifted convolution sums have been the focus of intense attention,
especially because of their applications to subconvexity problems and
to estimates for averages of $L$-functions. In many of the various
approaches to such sums, a key tool is their spectral decomposition.
Establishing such decompositions can be a difficult problem and
 various methods have been adopted to resolve it, each time
introducing an additional important aspect, e.g.
 estimates of triple products \cite{S}, the spectral structure of
$L^2(\Gamma \backslash G)$
($\Gamma=$PSL$_2(\mathbb Z)$, $G=$PSL$_2(\mathbb R)$)
\cite{BM},
Sobolev norms and Kirillov models \cite{BH}, multiple Dirichlet series
\cite{H}, etc.

Here we obtain the spectral decomposition of a shifted convolution sum
by using a combination of those approaches together with a new
element. The new element is the ``completion" of a generating
function of the convolutions sums into an square integrable function
on the modular curve which can be treated spectrally.

In slightly more detail, our shifted convolution sum is associated to
a cuspidal eigenform $f$ for $\Gamma_0(N)$, and is denoted by
$L(n, s, x; \chi)$ for
$n \in \mathbb Z\backslash \{0\}, x \in \mathbb R, s \in \mathbb C$
and a primitive Dirichlet character mod$N$, $\chi$. The overall
structure of the argument follows \cite{S} and \cite{BH}, but, in
order to obtain an object with a spectral expansion we consider a
series
(Eisenstein series formed with modular symbols) which naturally
encodes shifted convolution sums as its Fourier coefficients. This
series can then be ``completed" into an element of
$L^2(\Gamma \backslash \mathbb H, \chi)$ which can thus be analyzed
with spectral methods to derive the spectral expansion of
$L(n, s, x; \chi)$.

From the point of view of the theory of Eisenstein series with modular
symbols and, more generally, of higher-order forms, the significance
of our contribution is that, by connecting Fourier coefficients of a
higher-order form with shifted convolutions, we provide an arithmetic
meaning to those coefficients. That has been an aim of the theory of
higher order forms for some time because, in contrast to classical
modular forms, Fourier coefficients of forms of higher order did not
have an immediate arithmetic interpretation. Their number-theoretic
applications until now originated in their expression as Poincar\'e
series (\cite{G}, \cite{PR} etc.)

We first compute the Fourier coefficients of an Eisenstein series
modified with modular symbols in terms of a shifted convolution sum
$L(n, s; \chi)$. This is defined as the outcome of the analytic
continuation of a double Dirichlet series.

We then use the constructions of the first two sections to establish
the spectral decomposition of a weighted shifted convolution sum. An
application of this decomposition is the meromorphic continuation and
bounds for the convolution sum.

It is not clear that there is a pointwise spectral decomposition of
$L(n, s; \chi)$. As J. Hoffstein pointed out to us, $L(n, s; \chi)$
could be expressed in terms of the shifted convolution sum
$D'(w, v, m)$ studied in \cite{H}. Therefore, the spectral expansion
of $D'(w, v, m)$ established in Proposition 11.1 of \cite{H} might be
used to deduce the corresponding expansion of $L(n, s; \chi)$.
However it seems that the convergence of the sums involved is not
sufficiently uniform to allow to pass from the expansion of
$D'(w, v, m)$ to an expansion to $L(n, s; \chi)$.

\section{Eisenstein series formed with modular symbols} Let
$$f(z)=\sum_{n=1}^{\infty}a(n) e^{2 \pi i n z}$$
be a cusp form of weight $2$ for $\Gamma_0(N)$, with a positive
integer $N$. For every $c, d \in \mathbb Z$ ($c \ne 0$) we consider
the additive twist of the $L$-function of $f$, given for
$\mathrm{Re}(t)\gg 0$ by
$$ L(f, t; -d/c)=\sum_{n=1}^{\infty} \frac{a(n) e^{-2 \pi i
nd/c}}{n^t}.$$ We then set
$$\Lambda(f, t, -d/c):=\left (\frac{c}{2 \pi} \right )^t \Gamma(t)
L(f, t; -d/c)=c^t \int_0^{\infty} f(-d/c+ix)x^t \frac{dx}{x}.$$
The last expression can be used to give the analytic continuation of
$\Lambda(f, t, -d/c)$ to the entire $t$-plane for each
$\gamma=\left ( \begin{smallmatrix} a & b \\ c & d \end{smallmatrix} \right )
\in \Gamma_0(N)$. The function has a functional equation
$\Lambda(f, t, -d/c)=-\Lambda(f, 2-t, a/c),$ where
$ad \equiv 1 \mod c$ (see, e.g. \cite{KMV}, A.3)
which implies the convexity bound
\begin{equation}
\Lambda(f, t, -d/c)
\ll c^{3/2+ \epsilon} \qquad \text{for
$3/2+\epsilon >\mathrm{Re}(t)>1-\delta$ for some $\delta>0$}
\label{convexity}
\end{equation}
with the implied constant independent of $c$ and of $t$.

The modular symbol associated to the cusp form $f$ is:
$$<f, \gamma>:=\int_{i\infty}^{\gamma i \infty} f(w)dw \qquad
\text{for all $\gamma \in \Gamma_0(N)$.}$$
For each $\gamma = \left(\begin{smallmatrix} a&b\\c&d\end{smallmatrix}
\right)\in \Gamma_0(N)$ it satisfies
\begin{equation}
-<f, \gamma^{-1} >=-\int_{i\infty}^{-d/c} f(w)dw
=i\int_0^{\infty}f(-d/c+ iy)dy=\frac{i}{c}\Lambda(f, 1, -d/c)
\label{modsymb-addtw}
\end{equation}

With this notation we introduce the Eisenstein series at the cusp
$\infty$ with modular symbols. Let $\chi$ a singular character on
$\G=\Gamma_0(N)$. We set
\begin{equation} E(z, s; f, \chi):= \sum_{\g \in \G_\infty\backslash
\G} \overline{\chi(\g)}
<f, \g> \text{Im}(\gamma z)^s \label{esms}
\end{equation}
where, as usual, $\Gamma_\infty =
\bigl\{ \bigl(\begin{smallmatrix} \pm1&\ast\\0 & \pm1
\end{smallmatrix}\bigr) \bigr\}$ is the group of elements of $\G$
fixing $\infty$.

An explicit Fourier expansion of $E(\sb z, s;  f, \chi)$, and of more
general Eisenstein series with modular symbols, is essentially given
in \cite{C}, (1.1)-(1.3): Let for $c>0$, $c\equiv 0\bmod N$,
$$S(n, m, g, \chi ; c):=
\sum_{\g=\left ( \begin{smallmatrix} a & * \\ c & d
\end{smallmatrix}\right ) \in \G_{\infty} \backslash \G
/\G_{\infty}} \overline{\chi( \g )}
<f, \g > e^{2 \pi i (n \frac{d}{c}+m \frac{a}{c})}
$$
be the twisted Kloosterman sum. Then,
$$
E( z, s;  f, \chi)=\phi(s, f, \chi)y^{1-s}+\sum_{n \ne 0}
\phi(n, s, f, \chi)W_s(nz)
$$
with
\begin{eqnarray}
W_s(nz)&=&\sqrt{|n|y}K_{s-\frac12}(2 \pi |n|y)e^{2 \pi i n x}
\nonumber
\\
\label{0th}
\phi(s, f, \chi)&=&\sqrt {\pi} \frac{\G (s-\frac{1}{2})}{\G
(s)} \sum_{c >0,\, c\equiv0(N)}c^{-2s}S^\ast(0, 0, f, \chi ; c)
\\
\label{nth}
\phi(n, s, f, \chi)&=&\frac{\pi^s}{\G (s)}|n|^{s-1} \sum_{c >0,\,
c\equiv0(N)} c^{-2s}S^\ast(n, 0, f, \chi ; c)
\end{eqnarray}
Here $K_s(y)$ denotes the modified Bessel function.

We take as $\chi$ a character of $\Gamma_0(N)$ induced by a primitive
Dirichlet character $\mod N$ such that $\chi(-1)=1$. (Such a
character is singular at $\infty$). To $\chi$ we associate
$$\sigma_{t}^{\chi}(m):=\sum_{d|m}\chi(d)d^t
\quad \text{for each $m$ and} \,\,
W(\overline{\chi}):=\sum_{a \mod N} \overline{\chi}(a)
e^{2 \pi i a/N}$$
We also consider the Dirichlet $L$-function given for
$\mathrm{Re}(s)>1$ by
\begin{equation}
L(\overline{\chi}, s)=\sum_{n \ge 1, (n, N)=1}
\frac{\overline{\chi}(n)}{n^s}= \prod_{p \nmid N; \, \, \text{prime}}
\left
( 1-\frac{\overline{\chi}(p)}{p^{2s}}\right )^{-1}. \label{DirL}
\end{equation}

Since the cusp form $f$ we will be working with is fixed we omit it
from the notation, and write $E^\ast(z,s;\chi)=E(z,s;f,\chi)$, and
analogously $S^\ast(n,m,\chi;c)$, $\phi^\ast(s, \chi)$,
$\phi^\ast(n, s, \chi)$ etc.


To express the Fourier coefficients of $E^*$ in terms of more familiar
objects we will first need to analytically continue the double
Dirichlet series in \eqref{DDS} below.

\begin{prop} Fix an integer $n$ and an $s$ with $\mathrm{Re}(s)>2$.
\begin{enumerate}
\item[(i)] Then
\begin{equation}\sum_{l \ge 1, m \ne 0; l-m=n} \frac{a(l)
\sigma^{\chi}_{2s-1}(|m|) }{l^t |m|^{2s-1}}= \sum_{l \ge 1, l \ne n}
\frac{a(l)
\sigma^{\chi}_{2s-1}(|l-n|)}{l^t
(|l-n|)^{2s-1}} \label{DDS}
\end{equation}
is absolutely convergent for $\tau:=\mathrm{Re}(t)>3/2$ and has an
analytic continuation to $\mathrm{Re}(t)> 1/2 -\delta$ for some
$\delta>0$.

\item[(ii)] For these values of $s$ and $t$, \eqref{DDS} equals
\begin{equation}N^{2s} \frac{L(\bar \chi, 2s)}{W(\bar \chi)}\frac{(2
\pi)
^t}{\Gamma(t)} \sum_{N |c > 0} c^{-2s-t} \left (\sum_{d \mod c, (d,
c)=1} \overline{\chi(d)} e^{2 \pi i n \frac{d}{c}}\Lambda(f, t,
-\frac{d}{c})
\right )
\label{finalequation}
\end{equation}
\end{enumerate}
\label{analcont}
\end{prop}
\begin{proof} From $d(|a|) =o(|a|^{\epsilon})$,
 $|\sigma_{2s-1}^{\chi}(|a|)| \le d(|a|) |a|^{2\text{Re}(s)-1}$ and
 the Ramanujan bound we have
\begin{equation}
\left | \frac{a(l)
\sigma^{\chi}_{2s-1}(|m|) }{l^t |m|^{2s-1}} \right | \ll \left |
\frac{1 }{l^{t-1/2-\epsilon} m^{-\epsilon}} \right | \ll
\frac{1}{l^{\tau-1/2-2\epsilon}}. \label{bound}
\end{equation}
  Therefore the series converges absolutely when $\tau>3/2$.

To meromorphically continue this function, we will use a formula for
the Fourier coefficient of the non-holomorphic Eisenstein series.
Although this formula must certainly not be new we prove it here
because we have not found it in the literature in exactly this form.
\begin{lem} Let $N$ be a positive integer and $\chi$ be a primitive
character $\mod N$. Let $n \in \mathbb Z^*$ and $s$ with
$\mathrm{Re}(s)>1$. Then, the $m$-th coefficient $ \phi(m,s;\chi)$ of
$W_s(mz)$ in the Fourier expansion of
$$E(z, s; \chi)= \sum_{\gamma \in \Gamma_{\infty} \backslash \Gamma_0(N)}
\overline{ \chi(\gamma)} \text{Im}(\gamma z)^s$$
is
\begin{equation}
\phi(m,s;\chi) =\left ( \frac{\pi}{N^2}\right )^s \frac{W(\bar
\chi)}{\Gamma(s)} \frac{1}{L(\bar \chi, 2s)}
\frac{\sigma_{2s-1}^{\chi}(|m|)}{|m|^s} \label{Fcoeff}
\end{equation}
\label{Fcoefflemma}

\end{lem}
\noindent
\bf Proof of Lemma. \rm It is well-known that the $m$-th Fourier
coefficient of $E(z, s; \chi)$ is
\begin{equation}\frac{\pi^s |m|^{s-1}}{\Gamma(s)N^{2s}} \sum_{c>0}
c^{-2s} \sum_{0\leq d<Nc,\; (d,Nc)=1} \overline{\chi (d)}\, e^{2\pi i
|m| d/Nc}.\label{Kloo}
\end{equation}
To compute the double sum we follow the method of \cite{CoIw} (Section
3):
$$\sum_{0\leq d<Nc}
\sum_{\delta\mid (d,Nc)} \mu(\delta)\, c^{-2s}\, \overline{\chi (d)}\, e^{2\pi i
|m|d/Nc}=\sum_{d\geq 0}\sum_{\delta\mid d} \mu(\delta) \sum_{c> \frac{d}{N},\; \delta | Nc }
c^{-2s}\,\psi (d)\, e^{2\pi i|m|d/Nc}$$
$$\sum_{d\geq 0}\sum_{\delta\mid d} \mu(\delta)
\sum_{c:\; c\delta/(N,\delta)> d/N,\; \delta\mid c \frac\delta{(N,\delta)} N } \bigl(
c\delta/(N,\delta)\bigr)^{-2s}\, \overline{\chi (d)}\, e^{2\pi i |m| d (N,\delta)/\delta c
N}=$$
$$ \sum_{\delta\geq 1} \mu(\delta)
\,\Bigl(\frac{(N,\delta)}\delta\Bigr)^{2s}
\sum_{l \geq 0} \sum_{c>l(N,\delta)/N} c^{-2s}\, \overline{\psi (\delta l)}\, e^{2\pi
i|m|l \frac{(N,\delta)}{Nc}}.$$
Since $\chi(\delta l )=0$ if $(\delta, N) \ne 1$, this becomes
$$\sum_{\delta\geq 1, (N, \delta)=1} \frac{\mu(\delta)}{\delta^{2s}}
\sum_{d\geq 0} \sum_{c>l/N} c^{-2s}\, \overline{\chi(\delta l)}\, e^{\frac{2\pi
i|m|l }{Nc}}
=L(\bar \chi, 2s)^{-1}\sum_{c\geq 1}c^{-2s} \sum_{d=0}^{Nc-1} \overline{\chi (d)} \,
e^{2\pi i|m|d/Nc}\,.$$
Lemma 3.1.3 (1) of \cite{M} implies that
$$\sum_{c\geq 1}c^{-2s} \sum_{d=0}^{Nc-1} \overline{\chi (d)} \,
e^{\frac{2\pi i|m|d}{Nc}}=
\sum_{c | m}c^{-2s} \sum_{d=0}^{Nc-1} \overline{\chi (d)} \,
e^{\frac{2\pi i|m|d}{Nc}}=\sum_{c| n}c^{-2s+1} \sum_{d=0}^{N-1} \overline{\chi (d)} \,
e^{\frac{2\pi id \left ( \frac{|m|}{c}\right ) }{N}}
$$
Recall that $\chi$ is primitive $\mod N$ and that we have tacitly used
the same notation for the character $\mod Nc$ induced by $\chi$.
Lemma 3.1.1(1) in \cite{M} implies that the last expression is
$$\sum_{c | m}c^{-2s+1} W(\overline{\chi}) \chi(|m|/c)=W(\overline{\chi})
|m|^{1-2s}\sigma_{2s-1}^{\chi}(|m|).\qed$$

We proceed with the proof of the proposition. The lemma with $m=l-n$
combined with \eqref{Kloo} imply that, for $\mathrm{Re}(s)>1$ and
$\tau>3/2$ the series \eqref{DDS} equals
$$\frac{L(\overline{\chi}, 2s)}{W(\overline{\chi})}
 \sum_{l \ge 1, l \ne n} \frac{a(l)}{l^t} \sum_{c > 0}  c^{-2s}
\sum_{d \mod{N c }, \, \,  (d, N c)=1} \overline{\chi(d)} e^{2 \pi i |n- l|  \frac{d}{N c}}.$$
Since $\chi(-1)=1$, $|n-l|$ can be replaced by $n- l$ in the sum.
Further, the primitive character $\chi$ modulo $N > 1$ is not equal
to $1$ and hence we can omit the condition $l \ne n$. Since for
$\mathrm{Re}(s)>1$ and $\tau>3/2$,
$$ \sum_{l \ge 1} \frac{|a(l)|}{l^{\tau}} \sum_{c > 0}  c^{-2 \text{Re}(s)}
\sum_{d \mod{N c}, \, \, (d, N c)=1} \left | \overline{\chi(d)} e^{2 \pi i (n- l)  \frac{d}{N c}}\right | \ll
\sum_{l \ge 1} l^{1/2+\epsilon-\tau} \sum_{c > 0}  c^{-2 \text{Re}(s)+1}$$
converges uniformly, we can interchange the order of summation to get
$$\sum_{c > 0}  c^{-2s}
\sum_{d \mod{N c}, \, (d, N c)=1} \overline{\chi(d)} e^{2 \pi i n
\frac{d}{Nc}}
\sum_{l \ge 1}  \frac{a(l)}{l^t}e^{-2 \pi i  l \frac{d}{N c}} $$
With the definition of $\Lambda(f, t, -d/c)$ this implies
\eqref{finalequation}.

To prove the analytic continuation of \eqref{DDS} we note that each
$\Lambda(f, t, -d/c)$ has an analytic continuation to the entire
plane. Since it further satisfies \eqref{convexity}, the double sum
of \eqref{finalequation} is uniformly convergent for
$3/2+\epsilon>\mathrm{Re}(t)>1-\delta$
(and our fixed $s$ with $\mathrm{Re}(s)>2$)
giving an analytic function there. In the region
$\mathrm{Re}(t)>3/2+\epsilon$, our series is already analytic because
it is absolutely convergent there by the first part of the assertion.
\end{proof}

With \eqref{bound} we notice that, for $x>\mathrm{Re}(s)+1/2$, the
series
 $$\sum_{l \ge 1, m=l-n \ne 0} \frac{a(l)
\sigma^{\chi}_{2s-1}(|m|) }{l |m|^{s+x-1}}$$
is absolutely convergent. In view of this remark and Proposition
\ref{analcont} we make the definition
\begin{defn} Let $n \ne 0$ and consider a cusp form of weight $2$ for
$\Gamma_0(N)$
$$f(z)=\sum_{l=1}a(l)e^{2 \pi i l z}.$$
For each $s$ with $\mathrm{Re}(s)>2$ and $x>\mathrm{Re}(s)+1/2$ we
define $L(n, \chi, x; s)$ to be the value at $t=1$ of the analytic
continuation of
$$\sum_{l \ge 1, m=l-n \ne 0 }\frac{a(l)}{l^t}\frac{\sigma_{2s-1}^{\chi}(|m|)}{|m|^{2s-1}}
\left ( 1- \left |\frac{m}{n} \right |^{s-x}  \right ).$$
We also denote by $L(n, \chi; s)$ the value at $t=1$ of the analytic
continuation of
$$\sum_{l \ge 1, m=l-n \ne 0} \frac{a(l)
\sigma^{\chi}_{2s-1}(|m|) }{l^t |m|^{2s-1}}.$$
\end{defn}
For $n < 0$, $L(n, \chi; s)$ can be thought of as the "value at
$x=\infty$" of $L(n, \chi, x; s).$ Although these functions depend on
the function $f$, we do not include it in the notation to avoid
burdening it further.

\begin{thm} For $\mathrm{Re}(s)>2$ and $n \ne 0$,
we have
$$ L(n, \chi; s)=\frac{2\Gamma(s)N^{2s}L(\bar{\chi}, 2s)}
{i W(\bar{\chi}) (\pi |n|)^{s-1}}\phi^*(n, s, \chi).$$
The coefficient of $y^{1-s}$ in the Fourier expansion of
$E^*(z, s; \chi)$ is
\begin{equation}
\phi^*(s, \chi)=\frac{iW(\overline{\chi})\G(s-\frac12)}{2 \pi^{1/2}
N^{2s}\G(s)} \frac{L(f \otimes \chi, 1) L(f, 2s)}{L(\overline{\chi},
2s)
L(\chi, 2s)}. \label{constant}
\end{equation}
\label{FcoeffE*}
\end{thm}
\begin{proof}
We first observe that, because of \eqref{modsymb-addtw},
\begin{equation}
S^*(n, m, \chi; c)= \lim_{t \to 1} \frac{i}{c^t} \sum_{ad \equiv 1
\pmod c} \overline{\chi(d)} \Lambda(f, t, -d/c) \cdot
e^{2 \pi i (n \frac{d}{c}+m \frac{a}{c})}. \label{S^*}
\end{equation}
From \eqref{nth}, (ii) of Prop. \ref{analcont} and the definition of
$L(s, \chi, s)$ we get
\begin{multline*}
\phi^\ast(n, s, \chi)= \frac{\pi^s}{\Gamma(s)}|n|^{s-1}\lim_{t \to 1}
\sum_{c>0, N|c} ic^{-2s-t} \sum_{d \mod c, (d, c)=1}
\overline{\chi(d)} e^{2 \pi i n \frac{d}{c}}\Lambda(f, t,
-\frac{d}{c}) \\
= \frac{i \pi^s |n|^{s-1}}{\Gamma(s)}\lim_{t \to 1} \frac{W(\bar \chi)
\Gamma(t)}{L(\bar \chi, 2s) (2 \pi)^t N^{2s}} \sum_{l \ge 1, m=l-n \ne
0}\frac{a(l)
\sigma_{2s-1}^{\chi}(|m|)}{l^t |m|^{2s-1}}.
\end{multline*}
This implies the first part of the result.

Now, (ii) of Prop. \ref{analcont} can be applied to yield
\begin{equation}
\phi^*(s, \chi)=\frac{iW(\overline{\chi})
N^{-2s} \G(s-\frac12)}{2 \pi^{1/2} \G(s)
L(\overline{\chi}, 2s)} \lim_{t \to 1} \Big
(\sum_{l \ge 1} \frac{a(l)
\sigma_{2s-1}^{\chi}(l)}{l^{t+2s-1}}\Big )
\label{0coeff}
\end{equation}
Since $f$ is an eigenform of the Hecke operators we have, for each
$m, n \ge 1$ the identity
$$a(mn)=\sum_{d| (m, n)} \mu(d)d a(m/d)a(n/d)$$
and this implies that the sum in \eqref{0coeff}
$$\sum_{l, d, m \ge 1} \frac{\mu(l)l \chi(dl)a(d)
a(m)}{(dl)^t(ml)^{t+2s-1}}= \frac{L(f \otimes \chi, t) L(f, t+2s-1)}{L(\chi,
2t+2s-2)}.$$
Upon passing to the limit at $t=1$ we obtain the result.
\end{proof}
\bf Remark. \rm In many applications of convolution sums, it is
necessary to consider sums ranging over $m, l$ such that
$l_2 l-l_1 m=n$ for fixed integers $n, l_1, l_2$. These more general
sums can be also parametrized by Eisenstein series with modular
symbols as above and they have a spectral expansion. Here, we only
discuss the simpler case because the notation for general $l_1, l_2$
becomes more complicated and that would obscure the main point of our
construction.

\section{Spectral expansions} We will first derive a general
decomposition which will then be used to derive spectral expansions
of shifted convolution sums.

As before, we fix $N \in \mathbb N^*$ and a primitive character $\chi$
modulo $N$ and a cusp form
 $f(z)=\sum_{n=1}^{\infty} a(n)e^{2 \pi i n z}$ of weight $2$ for
 $\Gamma=\Gamma_0(N)$. Set
$$F(z)=\int_{i\infty}^z f(w)dw$$ and, for $\mathrm{Re}(s) >> 0$
$$G(z, s; \chi)=\sum_{\g \in
\G_{\infty}\backslash \G} \overline{\chi}(\g) F(\gamma z)
\text{Im}( \gamma z)^s.$$
We note that, for each $\gamma \in \Gamma_0(N)$ the modular symbol
satisfies
\begin{equation}
<f, \gamma^{-1}>= - <f, \gamma >
\qquad \text{and} \, \, <f, \gamma >=F(\gamma z)- F(z).
\label{modsymb-cobound}
\end{equation}
Therefore,
\begin{equation}
E^{*}(z, s; \chi)=G(z, s; \chi)-F(z)E(z, s; \chi). \label{G}
\end{equation}
The function $G$ has been studied in \cite{G} and shown to be
absolutely convergent for $\mathrm{Re}(s)>2$ and to belong to
$L^2(\Gamma \backslash \mathfrak H; \chi)$. As such, it is amenable
to a spectral expansion. Relation \eqref{G} represents the basis of
the ``completion" mentioned in the introduction.

We next consider a generalized Poincar\'e series which will allow us
to retrieve weighted shifted convolutions. Let $h$ be an element of
$C^{\infty}(0, \infty)$ which is $\ll y^{1/2-\epsilon},$ as
$y \to \infty$ and $\ll y^{1+\epsilon}$ as $y \to 0$ for some
$\epsilon>0.$ For $n \ne 0$, set
$$P(n, h, \chi; z):=\sum_{\gamma \in \Gamma_{\infty} \backslash \Gamma}
\overline{\chi} (\gamma)
e^{2 \pi i n \text{Re}(\gamma z)} h(2 \pi |n| \text{Im}(\gamma z))$$
A comparison with $E(z, s; \chi)$ shows that this is uniformly
absolutely convergent for $\mathrm{Re}(s)>2$ and that it belongs to
$L^{2}(\Gamma \backslash \mathfrak H; \chi).$

Let $k \in L^{2}(\Gamma \backslash \mathfrak H, \chi)$ have a Fourier
expansion
$$k(z)=\sum_{m \in \mathbb Z} c_m(y) e^{2 \pi i m x} \quad \text{with}
\, \, c_m(y)=\underset{y \to 0}{O}(y^A), =\underset{y \to \infty}{O}(y^B)
\qquad (-A, B <\epsilon).$$
Then, by unfolding the integral in the Petersson scalar product we
obtain:
\begin{equation}
\langle k, P(n, h, \chi)
\rangle=\int_0^{\infty} c_n(y)\overline{h(2 \pi |n| y)} \frac{dy}{y^2}
\label{fouriergener}
\end{equation}

\noindent \bf Parseval's formula \rm Since both series $G$ and $P$
defined above are in $L^{2}(\Gamma \backslash \mathfrak H, \chi)$, we
have
\begin{multline}\langle G, P(n, h, \chi)
\rangle=\sum_{j=1}^{\infty} \langle G, \eta_j \rangle \langle \eta_j,
P(n, h, \chi)
\rangle
+\\
\frac{1}{4\pi}\sum_{\mathfrak a} \int_{-\infty}^{\infty} \langle G,
E_{\ca}(\cdot, 1/2+ir, \chi) \rangle \langle E_{\ca}(z, 1/2+ir,
\chi), P(n, h, \chi)
\rangle dr \label{spd}
\end{multline}
 where
\begin{equation}
\eta_j(z)=\sum_{n \ne 0}b_j(n, \chi)W_{i r_j}(nz)
\label{b_j}
\end{equation} form a complete orthonormal basis of Maass cusp forms
with character $\chi$, with corresponding positive eigenvalues
$s_j(1-s_j) \to \infty$. The last sum ranges over a set of
inequivalent cusps and
\begin{equation}
E_{\ca}(z, s, \chi)=\delta_{\ca \infty}y^s+\phi_{\ca}(s, \chi)y^{1-s}+
\sum_{n \ne 0}\phi_{\ca}(n, s; \chi)W_s(nz)
\label{phi}
\end{equation} denotes the weight $0$ non-holomorphic Eisenstein
series at the cusp $\ca$. We write $s_j=\sigma_j+i r_j$ with
$\sigma_j \ge 1/2$ and $r_j \ge 0$. For almost all $j$,
$\sigma_j=1/2$.

To simplify notation, we will write $\phi_{\ca}(s), \phi_{\ca}(n, s)$
and $b_j(n)$ instead of $\phi_{\ca}(s, \chi), \phi_{\ca}(n, s; \chi)$
and $b_j(n; \chi).$

The first inner products in the series (resp. integral) are computed
essentially in \cite{G}. Let $L(f \otimes \eta_j, s)$ and
$L(f \otimes E_{\ca}(\cdot, \frac{1}{2}+ir, \chi), s)$
denote the Rankin-Selberg zeta function defined, for
$\mathrm{Re}(s) \gg 0$ by
$$\sum_{n=1}^{\infty} \frac{a(n)
\overline{b_j(n)}}{n^s} \quad
\text{and} \, \, \, \sum_{n=1}^{\infty} \frac{a(n)
\overline{\phi_{\ca}(n, \frac{1}{2}+ir)}}{n^s}$$
respectively. With this notation, we have
\begin{equation}
\langle G(\cdot, s; \chi), \eta_j \rangle =
\frac{\Gamma(s-\frac{1}{2}+ir_j)\Gamma(s-\frac{1}{2}-ir_j)} {(4\pi)^s
i\Gamma(s)}L(f \otimes \eta_j, s)
\label{inn1}
\end{equation}
and
\begin{equation}\langle G(\cdot, s; \chi), E_{\ca}(\cdot,
\frac{1}{2}+ir, \chi)\rangle =
\frac{\Gamma(s-\frac{1}{2}+ir)\Gamma(s-\frac{1}{2}-ir)} {(4\pi)^s
i\Gamma(s)}L(f \otimes E_{\ca}(\cdot,
\frac{1}{2}+ir, \chi), s). \label{inn2}
\end{equation}
The remaining inner products in the series are computed with
\eqref{fouriergener}:
$$ \langle \eta_j,  P(n, h, \chi) \rangle=
\sqrt{2 \pi} |n| b_j(n) (\mathcal K (\bar h) )(i r_j)$$
where
$$\mathcal K(h)(s):=\int_0^{\infty}K_{s}(y)h(y)\frac{dy}{y^{3/2}}.$$

\noindent \bf Fourier coefficients \rm By \eqref{G} and Lemma
\ref{Fcoefflemma} we see that the Fourier coefficient of
$e^{2 \pi i n x}$ in the expansion of $G(z, s; \chi)$ equals
\begin{multline}
\phi^*(n, s; \chi) \sqrt{|n| y} K_{s-\frac12}(2 \pi |n|y)+
\delta_{n>0} \frac{a(n) (y^s+\phi(s, \chi)y^{1-s})}{2 \pi i n e^{2
\pi n y}} +\\
\frac{W(\bar \chi) \pi^{s-1} \sqrt y}{2 i N^{2s} L(\bar \chi, 2s)
\Gamma(s)} \sum_{1 \le l, \ne n } \frac{a(l)}{l}\frac{K_{s-1/2}(2 \pi
|n-l|y)}{e^{2 \pi l y}}
\frac{\sigma_{2s-1}^{\chi}(|n-l|/k)}{|n-l|^{s-1/2}}
\end{multline}
Theorem \ref{FcoeffE*} implies that this Fourier coefficient equals
\begin{multline}
\frac{i W(\chi) (\pi |n|)^{s-1}}{2 \sqrt{2 \pi}\Gamma(s) N^{2s} L(\bar
\chi, 2s) }\Big
( L(n, \chi; s) \sqrt{2 \pi |n| y} K_{s-\frac12}(2 \pi |n|y)-\\
\delta_{n>0} \frac{2 \Gamma(s) N^{2s} L(\bar \chi, 2s)}{\sqrt{2 \pi}
W(\bar \chi) (\pi |n|)^{s-1}} \frac{a(n)}{n}\frac{(y^s+\phi(s,
\chi)y^{1-s})}{e^{2 \pi n y}} -\\
\frac{\sqrt{2 \pi |n|y}}{|n|^{s-1/2}} \sum_{l \ge 1, \ne n}
\frac{a(l)}{l}\frac{K_{s-1/2}(2 \pi |n-l|y)}{e^{2 \pi l y}}
\frac{\sigma_{2s-1}^{\chi}(|n- l|)}{|n-l|^{s-1/2}} \Big )
\end{multline}
Thus, integrating against $\bar h(2 \pi |n|y)/y^{2}$, we obtain with
\eqref{fouriergener},
\begin{multline}
\langle G, P(n, h, \chi) \rangle= \frac{i W(\bar \chi)
\pi^{s-1/2}}{\sqrt 2 \Gamma(s)
N^{2s}L(\bar \chi, 2s)} \biggl [ |n|^{s} L(n, s; \chi) \mathcal K
(\bar h)(s-1/2)- \\
\delta_{n>0} \frac{\sqrt{2} \Gamma(s) L(\bar \chi, 2s)}{W(\bar \chi)
\pi^{s-\frac12} N^{-2s}} a(n)
 \left ((2 \pi |n|)^{-s}\int_0^{\infty}y^{s-1}e^{-y}\bar
h(y)\frac{dy}{y}+
(2 \pi |n|)^{s-1}\phi(s)\int_0^{\infty}y^{-s} e^{-y} \bar
h(y)\frac{dy}{y} \right )\\
-\sum_{l \ge 1, \ne n}\frac{a(l)}{l}\frac{\sigma_{2s-1}^{\chi}
(|n-l|)}{|n-l|^{s-1}} \mathcal K \left
(\overline{h \left (\left | \frac{n}{n- l} \right | \cdot \right )}
e^{\frac{-l}{|n-l|} \cdot}\right )(s-1/2) \biggr ].
\end{multline}
The interchange of integration and infinite summation in the last term
is justified by the asymptotics of $K_{s-1/2}(y)$: $\ll y^{1/2-s}$ as
$y \to 0^+$ and $\ll e^{-y}y^{-1/2}$ as $y \to \infty$. Together with
the growth conditions of $h(y)$ at $0$ and $\infty$ and the bound
$e^{-x} \ll x^{-1-\epsilon}$, they imply that the $l$-th term of the
series is $\ll e^{-4 \pi l y} y^A l^B \ll e^{-4 \pi l} l^C $ for some
$A, B, C$ as $y \to \infty$ and $\ll l^{-1-\epsilon} y^{\epsilon}$ as
$y \to 0$. Therefore, the series converges uniformly in $(1, \infty)$
and in $(0, 1).$

Multiplying both sides with
$ i \Gamma(s)2^{2s-1/2} \pi^{s-\frac12}/|n|$ and taking into account
the definition of $L(n, \chi; s)$ we deduce the following
proposition.
(We use the following notational simplification:
$$\Gamma(a \pm b):=\Gamma(a+b) \Gamma(a-b).)$$
\begin{prop}\label{prop-spexp} Consider $s$ with $\mathrm{Re}(s)>2$
and an $h \in C^{\infty}(0, \infty)$ which is $\ll y^{1/2-\epsilon},$
as $y \to \infty$ and $\ll y^{\mathrm{Re}(s)+5/2 +\epsilon}$ as
$y \to 0$ for some $\epsilon>0$. If $n<0$, we have
\begin{multline}
\frac{-W(\bar{\chi})|n|^{s-1} }{N^{2s} L(\bar{\chi}, 2s)(2
\pi)^{1-2s}} \lim_{t \to 1} \Big ( \sum_{ l>0}
\frac{a(l)}{l^t}\frac{\sigma_{2s-1}^{\chi}(|n- l|)}{|n- l|^{2s-1}}
\mathcal K\left ( \bar h- \left |\frac{n-l}{n} \right |^{s} \bar h
\left (\frac{|n| \cdot}{|n- l|} \right )e^{-\frac{l \cdot }{|n- l|}}
\right )(s-\frac12)\Big )\\
=\sum_{j=1}^{\infty} \Gamma(s-\frac{1}{2} \pm ir_j) \mathcal K(\bar
h)(ir_j)
 b_j(n, \chi ) L(f \otimes \eta_j, s) +\\
\frac{1}{4\pi}\sum_{\mathfrak a} \int_{-\infty}^{\infty}
\Gamma(s-\frac{1}{2}\pm ir)
\mathcal K(\bar h)(ir) \phi_{\ca}(n, \frac12+ir; \chi) L(f\otimes
E_{\ca}(\cdot,
\frac{1}{2}+ir; \chi), s) dr. \label{4stform}
\end{multline}
If $n>0$ then the same identity holds with the term
$$\frac{\sqrt{2} \Gamma(s) a(n)
 (2 \pi)^{2s-1}}{\pi^{s-1/2} n}\left ((2 \pi |n|)^{-s}
 \int_0^{\infty}e^{-y}y^{s-1}\bar h(y)\frac{dy}{y}
 +(2 \pi |n|)^{s-1}\phi(s)\int_0^{\infty}y^{-s} e^{-y}\bar h(y)\frac{dy}{y} \right )$$
subtracted from the left hand side.
\end{prop}

This proposition allows us to derive the spectral decomposition for
the twisted shifted convolution sum we have been studying.
\begin{thm} Fix $n<0$ and $\mathrm{Re}(s)>2.$ For
$x>\mathrm{Re}(s)+ 5 /2$ we have
\begin{multline}
|n|^{s-1} L(n, \chi, x; s)=
-\left ( \frac{N}{2 \pi} \right )^{2s} \frac{2 \pi L(\bar{\chi},
2s)}{W(\bar{\chi})
\Gamma(s-1+x)\Gamma(x-s)} \times \\
\Big (\sum_{j=1}^{\infty} \Gamma(s-\frac{1}{2} \pm ir_j)
\Gamma(x-\frac{1}{2}\pm ir_j)b_j(n, \chi)
L(f \otimes \eta_j, s) +\text{cont. part} \Big )
\label{5stform}
\end{multline}
\label{weightedSD}
\end{thm}
\begin{proof}
We apply Proposition~\ref{prop-spexp} with the test function
$h(y)=h_x(y):= e^{-y}y^x$. It has the nice property that for
$\tilde h(y) = h_x\bigl(|n|y/|n-\nobreak l|\bigr)\, e^{-ly/{n-l|}}$
we have
$$( \mathcal K \tilde h )(u)
\;=\; \Bigl|\frac n{n-l}\Bigr|^x \, {\mathcal K}h_x(y)\,.$$
So the transform $\mathcal K$ occurring in the left hand side of
\eqref{4stform} is equal to
$ \bigl(1-\Bigl|\frac{n-l}n\Bigr| ^{s-x}\, \bigr){\mathcal K}h_x$.
The explicit form of the $K$-Bessel transforms appearing in the
formula are given by \cite{GR}, 6.621.3.
\end{proof}

An implication of this is the meromorphic continuation and bounds of
$L(n, \chi, x; s)$.
\begin{thm} For every integer $n < 0$, the function $L(n, \chi, z; s)$
 can be meromorphically continued to the entire $(s, z)$-plane. For
each $\epsilon>0$ and $z=x+iy$, $s=\sigma+it$ such that $x, \sigma
\in (0, \theta+ 3 )$:
\begin{multline}\label{mainbound}
\frac{\Gamma(s-1+z)\Gamma(z-s) \zeta(2s)}{ L(\bar{\chi}, 2s)} L(n,
\chi, z; s) \ll \\
|n|^{1-\sigma+\epsilon+\theta} \cdot
\begin{cases}
e^{-\pi|t|}\,\bigl(1+|t|\bigr)^{2(\theta+\epsilon)} \Bigl(
 e^{-\pi|y|/2}\,\bigl(1+|y|\bigr)^{2x+\epsilon-1} \,
\\
\qquad\qquad\qquad
+ e^{-\pi|t|/2}\, \bigl(1+|t|\bigr)^{2x+\epsilon-1}
\Bigr)
&\text{ if }|y|\leq|t|\,,\\
e^{-\pi|y|} \, \bigl(1+|y|\bigr)^{2(x-1)} \Bigl( e^{-\pi|t|/2} \,
\bigl(1+|t|\bigr)^{2\theta+3\epsilon+1}\\
\qquad\qquad\qquad
+ e^{-\pi|y|/2}\, \bigl(1+|y|\bigr)^{2\theta+3\epsilon+
1} \Bigr)&\text{ if }|t|\leq |y|\,,
\end{cases}
\end{multline}
where $\theta$ is the best exponent towards the Ramanujan conjecture
for Maass cusp forms and the implied constant depends on $N,$ $f,$
$\theta$ and $\epsilon$.

\end{thm}
\begin{proof} Each of the terms in the RHS of \eqref{5stform} is
meromorphic in $s$ and $z$. To show that the convergence of each of
the series/integrals is uniform on compacta we recall some estimates:
By (A16) of \cite{S1}
\begin{equation} \label{Dbound}
b_j(n)\ll e^{\pi r_j/2} |n|^{\epsilon+\theta}
\end{equation}
where $\theta$ is the best exponent towards the Ramanujan conjecture.

To bound the Rankin-Selberg zeta functions appearing in
\eqref{4stform} we first observe that, with \eqref{Dbound} and the
Ramanujan bound we have, for $\mathrm{Re}(s)>\theta+ 3 $ and
$\sigma_j=1/2$,
\begin{equation}
L(f \otimes \eta_j, s) \ll \sum_{n=1}^{\infty} \frac{n^{1/2+\epsilon}
n^{\epsilon+\theta}e^{\pi r_j/2}}{n^{\text{Re}(s)+1/2}}\ll e^{\pi
r_j/2}. \label{conv1}
\end{equation}
Then, with Stirling's estimate
we deduce that the normalized function
$$\Lambda (f \otimes \eta_j, s):=(2 \pi)^{-2s} \Gamma(s-1/2 \pm i r_j) \zeta(2s) L(f \otimes \eta_j, s)$$
is bounded by an absolute constant times
\begin{equation}
(2 \pi)^M \begin{cases} e^{-\pi |t|}
((1+|t|)^2-r_j^2)^{M-1}, \quad |t|>r_j
\\ e^{-\pi r_j} ((1+r_j)^2-|t|^2)^{M-1},
\quad |t|<r_j
\end{cases}
\label{Rbound}
\end{equation}
on $\mathrm{Re}(s)=M>\theta + 3 $.
To obtain the growth at $\mathrm{Re}(s)=1-M$ we recall the functional
equation of $L(f \otimes \eta_j, s)$
(e.g. \cite{DeIw}, Lemma 1) which, in our case can be written as:
\begin{equation}
\frac{\Gamma(s-1/2 \pm i r_j)}{(4 \pi)^s \Gamma(s)} \vec{L}(f \otimes
\eta_j, s)=\frac{\Gamma(1/2-s \pm i r_j)}{(4 \pi)^{1-s} \Gamma(1-s)}
\Phi(s)\vec{L}(f \otimes \eta_j, 1-s)
\label{RSf.e.}\end{equation}
where $\Phi(s)$ is the scattering matrix of the Eisenstein series and
$\vec{L}(s)$ is the column vector of Rankin-Selberg functions
$L_{\mathfrak{a_i}}$ of $f, \eta_j$ at the cusp $\mathfrak{a_i}$
(ranging over a set of $m$ inequivalent cusps of $\Gamma_0(N)$). For
our purposes, the only information about $L_{\mathfrak{a_i}}$
($\mathfrak{a_i} \ne \infty$) we need is their bounded-ness when
$\mathrm{Re}(s)>\theta+1$. We further use the formula for $ij$-entry
of the scattering matrix for $\Gamma_0(N)$:
$$\pi^{2s-3/2}\frac{\Gamma(1-s)}{\Gamma(s)}\frac{\zeta(2-2s)}{\zeta(2s)} P_{ij}(s).$$
Here $P_{ij}(s)$ is a finite product of $(p^s-p^{1-s})/(p^{2s}-1)$ or
$(p^{2s}-1)^{-1}$ with $p$ ranging over a set of primes fixed for a
given $N$ (cf. \cite{He}, pg. 535). With this formula, \eqref{RSf.e.}
implies that
\begin{equation}
\sqrt{\pi}\Lambda(f \otimes \eta_j, s)=\sum_{j=1}^m P_{\infty j}(s)
\Lambda_{\mathfrak{a_j}}(f \otimes \eta_j, 1-s). \label{lambdaf.e.}
\end{equation}
Since $(p^s-p^{1-s})/(p^{2s}-1)=O(p^{-s})$ and $(p^{2s}-1)^{-1}=O(1)$
for $\mathrm{Re}(s)=1-M<0$, \eqref{Rbound} and \eqref{lambdaf.e.}
imply that, for $\mathrm{Re}(s)=1-M$, there is some $a>0$ such that
\begin{equation}
\Lambda (f \otimes \eta_j, s) \ll e^{a M}(2 \pi)^M \begin{cases}
e^{-\pi |t|}
((1+|t|)^2-r_j^2)^{M-1}, \quad |t|>r_j
\\ e^{-\pi r_j} ((1+r_j)^2-|t|^2)^{M-1},
\quad |t|<r_j
\end{cases}
\label{Lbound}
\end{equation}
With Phragm\'en-Lindel\"of, this also gives a
bound for the strip:
$\mathrm{Re}(s) \in (-\theta-\epsilon  -2 , 3 +\theta+\epsilon)$,
\begin{equation}
\Lambda (f \otimes \eta_j, s) \ll
\begin{cases} e^{-\pi |t|}
((1+|t|)^2-r_j^2)^{\theta+\epsilon}, \quad |t|>r_j
\\ e^{-\pi r_j}
((1+r_j)^2-|t|^2)^{\theta+\epsilon}, \quad |t|<r_j. \end{cases}
\label{Cbound}
\end{equation}
Analogous bounds hold for
$L(f \otimes E_{\ca}(\cdot, \frac{1}{2}+ir, \chi), s)$.

We can now rewrite \eqref{5stform} in terms of $\Lambda$ in order to
use the bounds we just established.
$$L(n, \chi, z; s)
=\frac{-N^{2s} |n|^{1-s} 2 \pi L(\bar{\chi}, 2s)}{W(\bar{\chi})
\Gamma(s-1+z)\Gamma(z-s) \zeta(2s)}
\Big (\sum_{j=1}^{\infty} \Gamma(z-\frac{1}{2}\pm ir_j)b_j(n, \chi)
\Lambda(f \otimes \eta_j, s) +\text{c.~part} \Big ).$$
We first show that the sum on the right hand side converges uniformly
in compacta in $(s, z)$ to yield a meromorphic function in
$\mathbb{C}^2$. Indeed, for $(s, z)$ in a compact set $S$ not
containing any poles, equation \eqref{Lbound}, \eqref{Rbound} and
\eqref{Cbound} can be simplified as:
\begin{equation}
\Lambda (f \otimes \eta_j, s) \ll
\begin{cases} e^{-\pi |t|} (1+|t|)^B, \quad |t|>r_j
\\ e^{-\pi r_j} (1+r_j)^B, \quad |t|<r_j.
\end{cases}
\label{Genbound}
\end{equation}
with $B>0$ depending on $S$ only. Further, Stirling's estimate
implies:
\begin{equation}
\Gamma(z-1/2 \pm ir_j ) \ll \begin{cases}
e^{-\pi |y|} (1+|y|)^{2x-2} \qquad |y|>r_j
\\ e^{-\pi r_j}(1+r_j)^{2x-2} \quad |y| < r_j
\end{cases}
\label{Stir}
\end{equation}
 The Weyl law gives for the number of $j$ such that $|t_j|\leq T$ the
asymptotic formula $c_N T^2+ O_N(T\,\log T)$. So for
$l\in {\mathbb N}^\ast$ the number of $j$ with $l-1\leq |t_j|\leq l$
is bounded by $l^{1+\epsilon}$ for each $\epsilon>0$. Hence, for $(z, s)
= (x+iy,\sigma+it) \in S$:
\begin{multline} \sum_{j=1}^{\infty} \Gamma(z-\frac{1}{2}\pm i
r_j)b_j(n, \chi)
\Lambda(f \otimes \eta_j, s) \ll \sum_{l<\max{(|t|, |y|)}}
\Gamma(z-\frac{1}{2}\pm i l)b_j(n, \chi)
\Lambda(f \otimes \eta_j, s) + \\
\sum_{l>\max{(|t|, |y|)}} e^{-\pi l}
(1+l)^{2x-2} |n|^{\theta+\epsilon} e^{\frac{\pi l}{2}} l^2 e^{-\pi l}
(1+l)^B \ll \\ \sum_{l<\max{(|t|, |y|)}} \Gamma(z-\frac{1}{2}\pm i
l)b_j(n, \chi)
\Lambda(f \otimes \eta_j, s) + |n|^{\theta+\epsilon}
\sum_{l=1}^{\infty} e^{\frac{-3 \pi}{2} (l+\max{(|t|, |y|)})}
(1+l+\max{(|t|, |y|)})^C \label{discrete}
\end{multline}
where $C>1$ is a constant depending only on $S$. Since
$\max{(|t|, |y|)}$ is bounded by a constant depending on $S$ only and
$\sum_{l=1}^{\infty} e^{-3 \pi l/2}  (A+l)^C$ is convergent for
$A>0$, we deduce the required uniform convergence. The uniform
convergence of the integrals in the continuous part is similar.

To establish \eqref{mainbound} for $|t| \leq |y|$, we employ
\eqref{Dbound}, \eqref{Cbound} and \eqref{Stir}: 
\begin{multline*}
\sum_{j=1}^{\infty} \Gamma(z-\frac{1}{2}\pm ir_j)b_j(n, \chi)
\Lambda(f \otimes \eta_j, s) \ll \\
|n|^{\theta+\epsilon} \Big ( \sum_{l \leq |t|\leq|y|}
l^{1+\epsilon} e^{-\pi |y|}
(1+|y|)^{2x-2} e^{\pi l/2} e^{-\pi |t|} (1+|t|)^{2 \theta+2 \epsilon}
\\
+ \sum_{|t|<l \leq |y|} l^{1+\epsilon} e^{-\pi |y|}
 (1+|y|)^{2x-2} e^{\pi l/2} e^{-\pi l}
(1+l)^{2 \theta+2 \epsilon} 
\\
+\sum_{|t| \leq |y|<l} l^{1+\epsilon} e^{-\pi
l}
(1+l)^{2x-2} e^{\pi l/2} e^{-\pi l} (1+l)^{2 \theta+2 \epsilon} \Big )
 \\
\ll |n|^{\theta+\epsilon}\, \Bigl( e^{-\pi|y|}\,
\bigl(1+|y|\bigr)^{2x-2}\,\bigl(1+|t|\bigr)^{2(\theta+\epsilon)}
e^{-\pi|t|} \int_{u=0}^{|t|} e^{\pi u/2}\, u^{1
+\epsilon} \, du \\
+ e^{-\pi|y|}\,\bigl(1+|y|\bigr)^{2x-2}\, \int_{u=|t|}^{|y|} e^{-\pi
u/2}\,
(1+u)^{ 2\theta+3\epsilon+1}\, du\\
+ \int_{u=|y|}^\infty e^{-3\pi u/2}\,(1+u)^{2x+2\theta+
3\epsilon-1}\, du\biggr)
\\
\ll |n|^{\theta+\epsilon}\, \Bigl( 2 \times
e^{-\pi|y|-\pi|t|/2}\bigl(1+|y|\bigr)^{2x-2}
\bigl(1+|t|\bigr)^{2\theta+3\epsilon+1}
+ e^{-3\pi|y|/2} \, \bigl(1+|y|\bigr)^{2x
-1+2\theta+2\epsilon} \Bigr)\,.
\end{multline*}
Analogously we obtain the same bound for the continuous spectrum term,
and the bounds for the case $|y|\leq|t|$.
\end{proof}

\begin{cor}
 Fix an integer $n < 0$. For each $\epsilon \in (0, 1/4)$ and
$\sigma=\mathrm{Re}(s) \in (0, \theta+1/4)$ we have
\begin{equation}
\zeta(2s) L(n, \chi; s) \ll |n|^{5/2+\epsilon}
(1+|t|)^{\frac{ \epsilon+1}{2}-\sigma}.
\label{corbound}
\end{equation}
\end{cor}
\begin{proof}
Set $x:=\sigma+ 5 /2+\epsilon \in (0, \theta+ 3 )$. Then, by
definition, we have
$$L(n, \chi; s)=L(n, \chi, x; s)+|n|^{x-s}
\sum_{l \ge 1} \frac{a(l)}{l} \frac{\sigma_{2s-1}^{\chi}(l-n)}{(l-n)^{s-1+x}}.$$

To estimate the first term note the convexity bound for
$L(\bar \chi, 2s)$ in $\mathrm{Re}(s) \in (0, \theta+1)$:
\begin{equation}
L(\bar \chi, 2s) \ll_{\theta, \epsilon}
(1+|t|)^{-\sigma+(1+\epsilon)/2}, \label{Dirbound}
\end{equation}
and Stirling's estimate
$$(\Gamma(s-1+z) \Gamma(z-s))^{-1} \ll e^{\pi \max(|t|, |y|)} (1+|t+y|)^{3/2-\sigma-x} (1+|t-y|)^{\sigma+1/2-x}.$$
Then, with $y=0$ the bound \eqref{mainbound} in the theorem simplifies
as 
$$\zeta(2s) L(n, \chi, x; s) \ll |n|^{1-\sigma+\theta+\epsilon}
(1+|t|)^{ 2 \theta+\frac{ \epsilon-5}{2}-3\sigma}.$$
On the other hand, with $d(l)$ denoting the number of divisors
of $l$,
$$
\sum_{l \ge 1} \biggl|\frac{a(l)}{l}
\frac{\sigma_{2s-1}^{\chi}(l-n)}{(l-n)^{s-1+x}} \biggr|
\;\leq\; \sum_{l\geq 1} \frac{d(l)}{l^{1/2}}
\sum_{p\mid l-n} p^{1-2\sigma}\,
\frac1{(l-n)^{\frac52+\epsilon}}
 \;\ll\; \sum_{l\geq 1}\frac{d(l)}{l^{3+\epsilon}}\;=\;
O(1)\,.
$$ 
With \eqref{Dirbound} applied to $\zeta(2s)$ this leads to
\begin{equation}
\zeta(2s) L(n, \chi; s) \ll |n|^{1-\sigma+\theta+\epsilon} (1+|t|)^{2
\theta+\frac{ \epsilon- 5}{2}- 3\sigma}+|n|^{
5/2+\epsilon} (1+|t|)^{-\sigma+\frac{1+\epsilon}{2}}.
\label{boundL(s)}
\end{equation}
Since for $\sigma  \in (0, \theta+1/4)$ we have
$ 5/2+\epsilon>1-\sigma+\theta+\epsilon $ and
$-\sigma+\frac{1+\epsilon}{2}>
-3\sigma+2 \theta+\frac{ \epsilon- 5 }{2}$
we deduce the result. 
\end{proof}

\end{document}